\documentclass{amsart}

\usepackage{amssymb, amsmath}

\allowdisplaybreaks[4]

\newcommand{\pD}[2]{\frac{\partial #1}{\partial #2}}

\newcommand{\rD}[2]{\frac{d #1}{d #2}}

\newcommand{\vn}[1]{\lVert#1\rVert}

\newcommand{\IP}[2]{\left< #1 , #2 \right>}

\newcommand{\sfrac}[2]{\text{\fontsize{5}{5}\selectfont$\frac{#1}{#2}$}}

\newcommand{\R}{\ensuremath{\mathbb{R}}}
\newcommand{\N}{\ensuremath{\mathbb{N}}}

\newcommand{\SA}{\ensuremath{\mathcal{A{}}}}

\newtheorem{thm}[section]{Theorem}{\bf}{\it}
\newtheorem{cor}[section]{Corollary}{\bf}{\it}
\newtheorem{prop}[section]{Proposition}{\bf}{\it}
\newtheorem{lem}[section]{Lemma}{\bf}{\it}
\newtheorem*{rmk}{Remark}{\bf}{\rm}

\begin{document}

\title{Surface diffusion flow near spheres}
\author{Glen Wheeler}

\address{School of Mathematics and Applied Statistics,
         University of Wollongong,
         Northfields Ave,
         Wollongong, NSW 2500, Australia}


\begin{abstract}
We consider closed immersed hypersurfaces evolving by surface diffusion flow,
and perform an analysis based on local and global integral estimates.
First we show that a properly immersed stationary ($\Delta H \equiv  0$) hypersurface in $\R^3$ or
$\R^4$ with restricted growth of the curvature at infinity and small total tracefree curvature must
be an embedded union of umbilic hypersurfaces.
Then we prove for surfaces that if the $L^2$ norm of the tracefree curvature is globally initially
small it is monotonic nonincreasing along the flow.
We also derive pointwise estimates for all derivatives of the curvature assuming that its $L^2$ norm
is locally small.
Using these results we show that if a singularity develops the curvature must concentrate in a
definite manner, and prove that a blowup under suitable conditions converges to a
nonumbilic embedded stationary surface.
We obtain our main result as a consequence: the surface diffusion flow of a surface initially close
to a sphere in $L^2$ is a family of embeddings, exists for all time, and exponentially converges to
a round sphere.
\end{abstract}

\maketitle
\section{Introduction}

Let $f:M^n\times[0,T)\rightarrow\R^{n+1}$ be a family of compact immersed hypersurfaces $f(\cdot,t)
= f_t: M \rightarrow f_t(M) = M_t$ with associated Laplace-Beltrami operator $\Delta$, unit normal
vector field $\nu$, and mean curvature function $H$.  The surface diffusion flow
\begin{equation}
\label{SD}
\pD{}{t}f = (\Delta H)\nu,
\end{equation}
is the chief object of interest for this paper.  Equation \eqref{SD} is a fourth order degenerate
quasilinear parabolic evolution problem, the local existence of which is now standard in the
literature; see \cite{escher98surface} for example.  Our goal in this paper is to establish
regularity and stability results analagous to the pioneering work of Kuwert \& Sch\"atzle
\cite{kuwert2001wfs,kuwert2002gfw}.

The hallmark geometric characteristics of surface diffusion flow are the following:
using $\text{Vol }M_t$ to denote the volume enclosed by $M_t$ in $\R^{n+1}$ we compute
\begin{align}
\rD{}{t}\text{Vol }M_t &= \int_{M} \Delta Hd\mu = 0
\text{, and}
\label{eqvolcomp}
\\
\rD{}{t}\int_{M} d\mu &= \int_{M} H\Delta H d\mu = -\int_M |\nabla H|^2 d\mu \le 0;
\label{eqsurfareacomp}
\end{align}
so that a manifold evolving by \eqref{SD} will exhibit conservation of enclosed volume and monotonic
decreasing surface area.  Further, surface area is preserved exactly when the mean curvature of
$M_t$ is constant.  This is similar to mean curvature flow (with normal velocity $\partial^\perp_tf
= -H$) where surface area is monotonically decreasing, and stationary if $M_t$ is a minimal surface.
As mean curvature flow may be used as a tool in studying minimal surfaces, surface diffusion
flow may be used to study surfaces of constant mean curvature with a prescribed volume.  These
properties make surface diffusion flow one of the most natural fourth order flows one can consider
(with the other candidate being Willmore flow, where $\partial^\perp_tf = \Delta H + H|A^o|^2$), and
a model problem to be well studied before moving on to more general evolution equations.

Since one may derive \eqref{SD} by considering the $H^{-1}$-gradient flow for the area functional
(see Fife \cite{fife2000mps}), one may suspect that a program of study similar to that presented by
Kuwert \& Sch\"atzle \cite{kuwert2001wfs,kuwert2002gfw} on the gradient flow for the Willmore
functional is not approriate.  Indeed, the natural reaction is to guess that surface diffusion flow
has more in common with mean curvature flow than Willmore flow.  With this in mind, an immediate
goal is to establish a lower bound on the surface area of the evolving hypersurface.  This would
follow from the computation \eqref{eqvolcomp} and the isoperimetric inequality, so long as we can
show that if $f_0$ is an embedding, then $f_t$ is also an embedding.  The crucial point is to prove
(Proposition \ref{propMonotoneCurvature}) that under assumption
\eqref{eqS1maintheoremsmallnessrequirement} below the curvature is monotonically nonincreasing along
the flow.  Combining this with a result from Li \& Yau \cite{LY82} finally allows us to conclude
that embeddedness is preserved for surface diffusion flows satsifying
\eqref{eqS1maintheoremsmallnessrequirement}, and that the surface area is uniformly bounded away
from zero.

With our analysis in this paper we affirm that the general technique of localised integral estimates
and blowup analysis from \cite{kuwert2001wfs,kuwert2002gfw} is effective for surface diffuson flow.
Our main result is the following theorem, analagous to Theorem 5.1 in \cite{kuwert2001wfs}.

\begin{thm}
\label{thmS1maintheorem}
There exists an absolute constant $\epsilon_0>0$ such that if $f:M^2\times[0,T)\rightarrow\R^3$ is a
surface diffusion flow with 
\begin{equation}
\int_M|A^o|^2d\mu\bigg|_{t=0} \le \epsilon_0 < 8\pi
\label{eqS1maintheoremsmallnessrequirement}
\end{equation}
then $T=\infty$, $f:M^2\times[0,\infty)\rightarrow\R^3$ is a family of embeddings, and
$M_t\rightarrow S^2$ exponentially, where $S^2$ is a round sphere.
\label{maintheoremSDflow}
\end{thm}

Our larger aim is to instigate a systematic study of the asymptotic behaviour of the surface
diffusion flow.  One way to view the condition \eqref{eqS1maintheoremsmallnessrequirement} is that
the deviation of $f$ from being round is small in an averaged sense.  Our result here can then be
viewed as a kind of stability of spheres theorem in the $L^2$ norm.  It then becomes natural to
query on the global behaviour of $f$ when this deviation is greater than $\epsilon_0$, or even if it
is initially very large.  At this time, we do not know of any analytic example which rules out
dropping condition \eqref{eqS1maintheoremsmallnessrequirement}, or allows us to provide an upper
bound for $\epsilon_0$.  As a first step, it would be interesting to determine whether an immersed
sphere with image a symmetric figure 8, having zero enclosed volume, shrinks to a point and vanishes
in finite time.

The surface diffusion flow has been considered for some time in the literature.  First proposed by
Mullins \cite{mullins1957ttg} in 1957 (two years before he proposed the mean curvature flow), it was
originally designed to model the formation of thermal grooves in phase interfaces where the
contribution due to evaporation-condensation was insignificant.  Some time later, Davi, Gurtin, Cahn
and Taylor \cite{cahn1994sms,davi1990mpi} proposed many other physical models which give rise to
the surface diffusion flow.  These all exhibit a reduction of free surface energy and conservation
of volume; essential characteristics of surface diffusion flow.  There are also other motivations
for the study of \eqref{SD}.  For example, two years later Cahn, Elliot and Novick-Cohen
\cite{cahn1996che} proved that \eqref{SD} is the singular limit of the Cahn-Hilliard equation with a
concentration dependent mobility.  Among other applications, this arises in the modeling of
isothermal separation of compound materials.

Analysis of the surface diffusion flow began slowly. Baras, Duchon and Robert \cite{baras1984edi}
showed the global existence of weak solutions for two dimensional strip-like domains.  Later, Elliot
and Garcke \cite{elliott1997erd} analysed the surface diffusion flow of curves, and obtained local
existence and regularity for $C^4$-initial curves, and global existence for small perturbations of
circles.
Escher, Mayer and Simonett \cite{escher98surface} gave several numerical schemes for modeling
\eqref{SD}, and have also given the only two known numerical examples \cite{mayer2001nss} of
the development of a singularity: a tubular spiral and thin-necked dumbbell.
  Ito \cite{Ito1999sdf} showed that convexity will not be preserved under surface diffusion
flow, even for smooth, rotationally symmetric, closed, strictly convex initial hypersurfaces.
Quite recently, Blatt \cite{B10} generalised this and showed loss of convexity and loss of
embeddedness for a large class of flows.

There have been many important works on fourth order flows of a slightly different character,
from Willmore flow of surfaces to Calabi flow, a fourth order flow of metrics.  Significant
contributions to the analysis of these flows by the authors Kuwert, Sch\"atzle, Polden, Huisken,
Mantegazza and Chru\'sciel
\cite{chrusciel1991sge,kuwert2001wfs,kuwert2002gfw,mantegazza2002sge,polden:cas} are particularly
relevant, as the methods employed there are similar to ours here.

Simonett \cite{simonett2001wfn} used centre manifold techniques to show that the statement of
Theorem \ref{thmS1maintheorem} holds under the stronger assumption that $f_0$ is
$C^{2,\alpha}$-close to a round sphere.  Our analysis here is completely different, drawing
inspiration instead from the work of Kuwert \& Sch\"atzle \cite{kuwert2001wfs,kuwert2002gfw} on the
Willmore flow of surfaces.  The most transparent difference between surface diffusion and Willmore
flow is that one lacks the explicit structure of an $L^2$ gradient flow.
While we are able to show that under assumption \eqref{eqS1maintheoremsmallnessrequirement} we have
\[
\rD{}{t}\int_M|A^o|^2d\mu \le -\frac{1}{4}\int_M|\Delta H|^2d\mu,
\]
it is completely unknown whether or not this is true for initial data violating
\eqref{eqS1maintheoremsmallnessrequirement}.

This paper is organised as follows.
In Section 2 we derive integral estimates in the case where the tracefree curvature is locally small
in $L^n$, where $n\in\{2,3\}$.
Using these we conclude our first results, curvature estimates and a gap lemma on the stationary
solutions of \eqref{SD} satisfying a small tracefree curvature assumption and restricted growth of
the second fundamental form in $L^2$ at infinity.
The key ingredient allowing the classification result to go through in three intrinsic dimensions is
a new multiplicative Sobolev inequality.
We also outline the proof of a lifespan theorem and interior estimates,
which are analogous to Theorem 1.2 in \cite{kuwert2002gfw} and Theorem 3.5 in \cite{kuwert2001wfs}.
Section 3 contains the blowup construction we will employ, and shows that for surface diffusion
flows satisfying \eqref{eqS1maintheoremsmallnessrequirement} the blowup at a finite time curvature
singularity is a stationary, nonumbilic surface with $\Delta H \equiv 0$.
In Section 4 we show that this is in contradiction with the gap lemma, and thus obtain long time
existence for surface diffusion flows satisfying \eqref{eqS1maintheoremsmallnessrequirement}.
In Section 4 we also examine the global behaviour of such a flow and show exponential convergence to
spheres.

We have developed the exposition to be particularly relevant to the case of globally constrained
surface diffusion and Willmore flows, where the immersion evolves by
\begin{equation}
\pD{}{t}f = (\Delta H + h)\nu,\quad\text{and}\quad\pD{}{t}f = (\Delta H + H|A^o|^2 + h)\nu
\label{eqnS1csdflow}
\end{equation}
respectively, for some constraint $h:[0,T)\rightarrow\R$.  The flow is non-local, in the sense that
the motion of points on $M$ may depend upon global properties of $M$, such as total curvature,
surface area, mixed volumes, or other quantities.  Choices of $h$ are motivated with some geometric
consideration in mind, and give rise to mixed volume preserving surface diffusion flow, and surface
area, volume preserving Willmore flows for example.  With suitable structure conditions placed upon
the constraint function $h$, one may conclude analagous results to those presented here,
albeit with the distance from a sphere in a possibly higher $L^p$ norm \cite{mythesis,MWW11}.

\section*{Acknowledgements}

This work forms part of the author's PhD thesis under the support of an Australian Postgraduate
Award.  Part of this research was completed during two visits to the Freie Universit\"at in Berlin.
The first visit was under the support of the Deutscher Akademischer Austausch Dienst, and the second
supported by the Research Group in Geometric Analysis at the Freie Universit\"at under the
supervision of Prof. Dr. Klaus Ecker.  Part of this work was also completed during a visit to the
Max Planck Institut in Golm, under the supervision of Prof. Dr. Gerhard Huisken.  He is grateful
for their support and hospitality.

The author is also very grateful to the referee who carefully read the manuscript and suggested
numerous improvements and simplifications to the arguments and exposition presented in this paper.

\section{Local estimates}

This section is in two parts.  In the first part we compute integral estimates which hold for any
immersion, and use these with multiplicative Sobolev inequalities to show curvature estimates and a
gap lemma.
In the second part our primary focus is on proving estimates which will allow us to overcome our
lack of a convenient gradient flow.  There we utilise the evolution of certain integral quantities
along the flow to conclude interior estimates and a lifespan theorem, similar to Theorems 3.5 and
1.2 in \cite{kuwert2001wfs} and \cite{kuwert2002gfw} respectively.

Throughout our analysis we will make extensive use of a cutoff function $\gamma=\tilde{\gamma}\circ
f:M\rightarrow[0,1]$ with compact support in $M$ and $\tilde{\gamma} \in C^2_c(\R^3)$ which
satisfies
\begin{equation}
\vn{\nabla\gamma}_\infty \le c_{\gamma1},\quad
\vn{\nabla_{(2)}\gamma}_\infty \le c_{\gamma2}(1+|A|),
\label{eqnS2eqngamma}
\end{equation}
for some absolute constants $c_{\gamma1}, c_{\gamma2} < \infty$.
Note $\gamma$ automatically has compact support in $M$ so long as $f$ is a proper immersion.
We frequently use the notation $\vn{\cdot}_p$ for the $L^p$ norm, and $\vn{\cdot}_{p,[\gamma>0]}$
for the $L^p$ norm restricted to the set where a function $\gamma$ is positive.  We also use
$P_i^j(T)$ to denote a sum of contractions of $i$ copies of $T$ with $j$ covariant derivatives in
each contraction.


Since we feel this argument is quite robust, we will use the convention $F = \partial^\perp_tf =
\Delta H$ to indicate how one may proceed for other flows.  We begin the first part of this section
by noting some easy consequences of results from \cite{kuwert2001wfs}.

\begin{prop}
For an immersed hypersurface
$f:M^n\rightarrow\R^{n+1}$, $\gamma$ as in \eqref{eqnS2eqngamma} and $s\ge4$ we have
\begin{align*}
\int_M&|\nabla_{(2)}A|^2\gamma^sd\mu
 +  \int_M|A|^2|\nabla A|^2\gamma^sd\mu
 +  \int_M|A|^4|A^o|^2\gamma^sd\mu
\\&
\hskip-5mm\le c\int_M |F|^2\gamma^sd\mu
      + c{(c_{\gamma1})}^4\int_{[\gamma>0]}|A|^2d\mu
      + c_1\int_M (|A^o|^2|\nabla A^o|^2+|A^o|^6)\gamma^sd\mu,
\end{align*}
where $c$ and $c_1$ are functions of $n$ and $s$ only.
\label{propS2ks1}
\end{prop}
\begin{proof}
Combine the obvious estimate
\begin{equation}
\int_M H^2|A^o|^4\gamma^sd\mu
 \le \delta\int_M H^4|A^o|^2\gamma^sd\mu + \frac{1}{4\delta}\int_M |A^o|^6\gamma^sd\mu,
\label{eqnS2simpleest}
\end{equation}
valid for any $\delta > 0$, and Proposition 2.4 from \cite{kuwert2001wfs}.
\qed
\end{proof}

The next four lemmas are generalisations of results from \cite{kuwert2001wfs,kuwert2002gfw} to
at least include the case $n=3$.
We include the first three with a view toward future applications, while the fourth is instrumental
in proving a gap lemma for $n=3$.
Each proof relies in a crucial way on the exponent in the
Michael-Simon Sobolev inequality \cite{michael1973sam}, and for $n\ge4$ we have not been able to
obtain a statement as useful as \eqref{msin3}.
We will omit the proof of the first three lemmas as we feel they are quite straightforward.

\begin{lem}
\label{myLZthm}
Let $f:M^n\rightarrow\R^{n+1}$ be an immersed hypersurface. For $u\in C_c^1(M)$,
$n<p\le\infty$, $0\le \beta\le \infty$ and $0<\alpha\le 1$ where $\frac{1}{\alpha} =
\big(\frac{1}{n}-\frac{1}{p}\big)\beta + 1$ we have
\begin{equation*}
  \vn{u}_\infty \le c\vn{u}_\beta^{1-\alpha}(\vn{\nabla u}_p + \vn{Hu}_p)^\alpha,
\end{equation*}
where $c = c(n,p,\beta)$.
\end{lem}
\begin{lem}
Suppose $n\in\{2,3\}$ and let $\gamma$ be as in \eqref{eqnS2eqngamma}.
Then for any tensor $T$ on $M^n$ and $s\ge2$ there is a $c=c(n,s)$ such that
\begin{align*}
\vn{T\gamma^s}^4_{\infty}
\le
c\vn{T\gamma^s}^{4-n}_2\bigg[
\Big(\int_M |\nabla_{(2)}T|^2\gamma^{2s}d\mu\Big)^\frac{n}{2}
 &+ \Big(\int_M H^4|T|^2\gamma^{2s}d\mu\Big)^\frac{n}{2} 
\\
&\hskip-7mm
\qquad+(c_{\gamma1})^{2n}\Big(\int_{[\gamma>0]}|T|^2d\mu\Big)^\frac{n}{2}
\bigg].
\end{align*}
\label{CieL6}
\end{lem}
\begin{lem}
\label{lemS2int}
Let $n\in\{2,3\}$.  Then for any tensor $T$ on $f:M^n\rightarrow\R^{n+1}$ and $\gamma$
as in \eqref{eqnS2eqngamma}
there is a $c=c(c_{\gamma1},n)$ such that
\begin{equation*}
  \vn{T}^4_{\infty,[\gamma=1]}
    \le c\vn{T}^{4-n}_{2,[\gamma>0]}\big( \vn{\nabla_{(2)}T}^n_{2,[\gamma>0]}
                                   + \vn{TA^2}^n_{2,[\gamma>0]}
                                   + \vn{T}^n_{2,[\gamma>0]}\big).
\end{equation*}
\end{lem}
\begin{lem}
\label{lemS2msiboth}
Suppose $\gamma$ is as in \eqref{eqnS2eqngamma}.  Then for an immersion
$f:M^2\rightarrow\R^3$ and $s\ge4$,
\begin{align}
&\int_M|A^o|^6\gamma^sd\mu + \int_M|A^o|^2|\nabla A^o|^2\gamma^sd\mu
\notag\\*
   &\quad\le  c_2\int_{[\gamma>0]}|A^o|^2d\mu
	  \int_M\big(|\nabla_{(2)}A^o|^2 + |A|^2|\nabla A^o|^2 +
                     |A|^2|A^o|^4\big)\gamma^sd\mu
\notag\\*
   &\hskip+2cm + c{(c_{\gamma1})}^4\Big(\int_{[\gamma>0]}|A^o|^2d\mu\Big)^2,\text{ and}
\label{eqnS2msi1Azero}
\\
&\int_M|A|^6\gamma^sd\mu + \int_M|A|^2|\nabla A|^2\gamma^sd\mu
\notag\\
&\quad\le  c\int_{[\gamma>0]}|A|^2d\mu\int_M(|\nabla_{(2)}A|^2 + |A|^6)\gamma^sd\mu
         + c(c_{\gamma1})^4\Big(\int_{[\gamma>0]}|A|^2d\mu\Big)^2,
\label{eqnS2msi2A}
\end{align}
and for an immersion $f:M^3\rightarrow\R^4$ and $s\ge6$,
\begin{align}
&\int_M|A^o|^6\gamma^sd\mu + \int_M|A^o|^2|\nabla A^o|^2\gamma^sd\mu
\notag\\*
 &\le
\delta\int_M|\nabla_{(2)}A^o|^2\gamma^sd\mu
+
 c_3\vn{A^o}_{3,[\gamma>0]}^\frac{3}{2}
     \int_M \big(|\nabla_{(2)} A^o|^2 + |A|^2|A^o|^4+|A^o|^6\big) \gamma^s d\mu
\notag\\
&\qquad
    + c(c_{\gamma1})^3\big(\vn{A^o}^3_{3,[\gamma>0]} + \vn{A^o}^\frac{9}{2}_{3,[\gamma>0]}\big),
\label{msin3}
\end{align}
where $\delta>0$ and $c = c(s,n)$.
\end{lem}
\begin{proof}
As the $n=2$ case is Lemma 2.5 in \cite{kuwert2001wfs} and Lemma 4.2 in \cite{kuwert2002gfw}, we
only prove the $n=3$ case.
First, using integration by parts we estimate
\begin{equation}
\int |\nabla A^o|^3\gamma^sd\mu 
\le \delta\int_M |\nabla_{(2)}A^o|^2\gamma^sd\mu + c_\delta\int_M |A^o|^6\gamma^sd\mu
   + c(c_{\gamma1})^3\int_{[\gamma>0]} |A^o|^3d\mu,
\label{eqS2n31}
\end{equation}
for any $\delta > 0$.
Now we use the Michael-Simon Sobolev inequality \cite{michael1973sam} with
$u=|A^o|^4\gamma^{2s/3}$ to estimate
\begin{align*}
\Big(&\int_M |A^o|^6 \gamma^sd\mu\Big)^\frac{2}{3}
\le c\int_M \big|\nabla \big(|A^o|^4\gamma^{2s/3}\big)\big| d\mu
    + c\int_M |H|\cdot|A^o|^4\gamma^{2s/3} d\mu 
\\
&\le c\int_M |\nabla A^o|^2|A^o|\gamma^sd\mu + 
                \Big(\int_M|A^o|^6\gamma^sd\mu\Big)^\frac{2}{3}
                \Big(\int_{[\gamma>0]}|A^o|^{3}\Big)^\frac{1}{3}
\\*
&\qquad
    + c\Big(\int_{[\gamma>0]}|A^o|^3d\mu\Big)^\frac{1}{3}
       \Big(\int_M |A|^2|A^o|^4\gamma^{s} d\mu\Big)^\frac{2}{3}
    + c(c_{\gamma1})^2\vn{A^o}^3_{3,[\gamma>0]}.
\end{align*}
Note that we needed $s\ge6$.  Estimating the above and then combining the result with
\eqref{eqS2n31} gives
\begin{align*}
\int_M |A^o|^6 \gamma^s d\mu
 &\le c\vn{A^o}_{3,[\gamma>0]}^\frac{3}{2}\int_M |\nabla A^o|^3\gamma^sd\mu
    + c(c_{\gamma1})^3\vn{A^o}^\frac{9}{2}_{3,[\gamma>0]}
\\*
&\qquad
    + c\vn{A^o}_{3,[\gamma>0]}^\frac{3}{2}\int_M \big(|A^o|^6+|A|^2|A^o|^4\big) \gamma^s d\mu
\\
 &\le c\vn{A^o}_{3,[\gamma>0]}^\frac{3}{2}
     \int_M \big(|\nabla_{(2)} A^o|^2 + |A|^2|A^o|^4+|A^o|^6\big) \gamma^s d\mu
\\*
&\qquad
    + c(c_{\gamma1})^3\vn{A^o}^\frac{9}{2}_{3,[\gamma>0]}.
\end{align*}
This estimates the first term.  For the second, simply estimate
\begin{align*}
\int_M|A^o|^2|\nabla A^o|^2\gamma^sd\mu
 &\le c\int_M |A^o|^6 \gamma^s d\mu + c\int_M |\nabla A^o|^3\gamma^s d\mu
\end{align*}
and again use \eqref{eqS2n31}.
This estimates the second term, and combining the two estimates above finishes the proof.
\qed
\end{proof}
We can now prove our main estimate for this part.

\begin{prop}
\label{propS2GLprop5}
Suppose $n\in\{2,3\}$, $\gamma$ is as in \eqref{eqnS2eqngamma} and $f:M^n\rightarrow\R^{n+1}$ is an
immersion.
Then there exists an $\epsilon_0 > 0$ such that if $\vn{A^o}^n_{n,[\gamma>0]}<\epsilon_0$
we have 
\begin{align*}
\int_M&\big(|\nabla_{(2)}A|^2
 +  |A|^2|\nabla A|^2
 +  |A|^4|A^o|^2\big)\gamma^6d\mu
\\*&
\le c\int_M |F|^2\gamma^6d\mu
      + c{(c_{\gamma1})}^4
        \int_{[\gamma>0]}|A|^2d\mu
    + c(c_{\gamma1})^{6-n}\vn{A^o}^n_{n,[\gamma>0]},
\end{align*}
where $c$ depends on $n$ only.
\end{prop}
\begin{proof}
First note that
\begin{equation}
\vn{A^o}^4_{2,[\gamma>0]} \le \epsilon_0 \vn{A^o}^2_{2,[\gamma>0]},
\ \ \text{ and }\ \ 
\vn{A^o}^\frac{9}{2}_{3,[\gamma>0]} \le \epsilon^\frac{1}{2}_0 \vn{A^o}^3_{3,[\gamma>0]},
\label{eqS2glp1}
\end{equation}
for $n=2$ and $n=3$ respectively.
Combining Proposition \ref{propS2ks1} with \eqref{eqnS2msi1Azero} for $n=2$, \eqref{msin3} for
$n=3$, gives
\begin{align*}
  \int_M &\big(|\nabla_{(2)}A|^2 + |\nabla A|^2|A|^2 + |A|^4|A^o|^2 \big)\gamma^sd\mu
\\
 &\le c\int_M |F|^2\gamma^sd\mu
    + c{(c_{\gamma1})}^4\vn{A}^2_{2,[\gamma>0]}
    + c_1\delta(n-2)\int_M|\nabla_{(2)}A|^2\gamma^sd\mu
\\&\quad
    + c_1c_n\vn{A^o}_{n,[\gamma>0]}^\frac{n}{n-1}
      \int_M \big(|\nabla_{(2)}A|^2 + |\nabla A|^2|A|^2 + |A|^4|A^o|^2\big) \gamma^s d\mu
\\
&\quad
    + c(n-2)(c_{\gamma1})^3\big(\vn{A^o}^3_{3,[\gamma>0]} + \vn{A^o}^\frac{9}{2}_{3,[\gamma>0]}\big)
\\
&\quad
    + c(3-n)(c_{\gamma1})^4\vn{A^o}^4_{2,[\gamma>0]},
\end{align*}
for any $\delta>0$.
Absorbing and using \eqref{eqS2glp1},
\begin{align*}
 (1-c_1c_n&\epsilon_0^\frac{1}{n-1} - c_1\delta(n-2))
 \int_M\big(|\nabla_{(2)}A|^2 + |\nabla A|^2|A|^2 + |A|^4|A^o|^2\big)\gamma^6d\mu
\\
 &\le c\int_M|F|^2\gamma^6d\mu
    + c{(c_{\gamma1})}^4\vn{A}^2_{2,[\gamma>0]}
    + c(c_{\gamma1})^{6-n}\vn{A^o}^n_{n,[\gamma>0]}.
\end{align*}
Therefore, for $\delta, \epsilon_0$ sufficiently small we may conclude the result.
\qed
\end{proof}

We can now give our main theorem for this part.

\begin{thm}[Gap Lemma]Suppose $n\in\{2,3\}$ and $f:M^n\rightarrow\R^{n+1}$ is a properly immersed
surface with $F = \Delta H \equiv 0$.  Then if
\begin{equation*}
\int_M |A^o|^n d\mu < \epsilon_0,\quad\text{ and }\quad
\liminf_{\rho\rightarrow\infty} \frac{1}{\rho^4} \int_{f^{-1}(B_\rho(0))} |A|^2 d\mu = 0,
\end{equation*}
$f$ is an embedded plane or sphere.
\label{thmS2gaplemma} 
\end{thm}
\begin{proof}
We set the cutoff function $\gamma$ to be such that $\gamma(p) =
\varphi\left(\frac{1}{\rho}|f(p)|\right)$, where $\varphi\in C^1(\R)$ and
$\chi_{B_{1/2}(0)} \le \varphi \le \chi_{B_1(0)}$.
With this choice ${c_{\gamma1}} = \frac{c}{\rho}$.  Recall that in our estimates we do not use the
second derivative of $\gamma$.  Taking $\rho\rightarrow\infty$ in Proposition \ref{propS2GLprop5}
gives
\[
  \int_M |\nabla_{(2)}A|^2 + |\nabla A|^2|A|^2 + |A|^4|A^o|^2d\mu
 \le c\liminf_{\rho\rightarrow\infty}\frac{1}{\rho^4}\vn{A}^2_{2,{f^{-1}(B_\rho(0))}},
\]
where the terms involving $A^o$ vanished since $\vn{A^o}^n_2 < \epsilon_0 < \infty$.
Therefore $|A| = 0$, $|A^o|=0$, or both, which implies that $f$ maps into a round sphere or plane
$S\subset\R^{n+1}$, and since $f$ is complete the map $f:M\rightarrow S$ is a global isometry.
\qed
\end{proof}


We now begin the second part of this section, changing focus by proving estimates which rely on the
parabolic nature of the evolution equation \eqref{SD}. To begin, we state the following elementary
evolution equations, whose proof is standard \cite{huisken1999gee}.
\begin{lem}
\label{lemS2evolutionequations}
For a surface diffusion flow $f:M^n\times[0,T)\rightarrow\R^{n+1}$ the following equations hold:
\begin{align*}
  \pD{}{t}g_{ij} &= 2FA_{ij},\quad
  \pD{}{t}g^{ij} = -2FA^{ij},\quad
  \pD{}{t}d\mu = (HF)d\mu,\\
  \pD{}{t}\nu &= -\nabla F,\quad
  \pD{}{t}A_{ij} = -\nabla_{ij}F + FA_i^pA_{pj},\\
  \pD{}{t}A^o_{ij} &= -S^o(\nabla_{(2)}F) + F\big[A_i^pA_{pj}
                                          + \sfrac{1}{n}g_{ij}|A|^2-\sfrac{2}{n}HA_{ij}\big],
  \quad\text{ and }\\
  \pD{}{t}\nabla_{(k)}A &= -\Delta^2\nabla_{(k)}A + P_3^{k+2}(A),\ k\in\N_0,
\end{align*}
where $S^o(T)$ is the symmetric tracefree part of a bilinear form $T$.
\end{lem}
Using the above evolution equations, and the previous estimates, we may now prove that the Willmore
energy is nonincreasing for surface diffusion flows satisfying
\eqref{eqS1maintheoremsmallnessrequirement}.

\begin{prop}Suppose $f:M^2\times[0,T)\rightarrow\R^3$ is a surface diffusion flow satisfying
\eqref{eqS1maintheoremsmallnessrequirement}.  Then
\[
\rD{}{t}\int_M|A^o|^2d\mu = \frac{1}{2}\rD{}{t}\int_M|H|^2d\mu
                          \le -\frac{1}{4}\int_M|\Delta H|^2d\mu,
\]
and $f(\cdot,t)$ is a family of embeddings.
\label{propMonotoneCurvature}
\end{prop}
\begin{proof}
Proposition \ref{propS2GLprop5} with $\gamma \equiv 1$ gives
\begin{equation*}
\int_M|\nabla_{(2)}A|^2 + |A|^2|\nabla A|^2 + |A|^4|A^o|^2 d\mu
\le c\int_M |\Delta H|^2d\mu.
\label{eqS2mono1}
\end{equation*}
Combining this with \eqref{eqnS2simpleest} and \eqref{eqnS2msi1Azero} we obtain
\begin{align}
\int_M|A^o|^4H^2 d\mu
&\le c\delta\int_M |\Delta H|^2d\mu + c_\delta\int_M |A^o|^6 d\mu
\notag\\*
&\le c(\delta + c_\delta\epsilon_0)\int_M |\Delta H|^2d\mu.
\label{eqS2mono2}
\end{align}
Using Gauss-Bonnet, Lemma \ref{lemS2evolutionequations} and \eqref{eqS2mono2} we compute
\begin{align*}
\rD{}{t}&\int_M|A^o|^2d\mu = \frac{1}{2}\rD{}{t}\int_M|H|^2d\mu
\\
&= - \int_M (\Delta H + H|A^o|^2)(\Delta H)d\mu
\\
&\le - \frac{1}{2}\int_M |\Delta H|^2 d\mu
     + c\int_M|A^o|^4H^2d\mu
\\
&\le - \frac{1}{4}\int_M |\Delta H|^2 d\mu,
\end{align*}
for $\delta,\epsilon_0$ sufficiently small.  We have thus shown, for all $t$,
\[
\frac{1}{4}\int_MH^2d\mu < 8\pi,
\]
and it follows from Theorem 6 in \cite{LY82} that each $f(\cdot,t)$ is an embedding.
\qed
\end{proof}

\begin{lem}
\label{lemS2gradests1}
Let $\gamma$ be as in \eqref{eqnS2eqngamma}.  The following equalities hold
for a surface diffusion flow $f:M^n\times[0,T)\rightarrow\R^{n+1}$:
\begin{align}
\rD{}{t}\int_M\frac{1}{2}H^2\gamma^sd\mu + \int_M|F|^2\gamma^sd\mu
  &=
   \frac{1}{2}\int_M H^2(\partial_t\gamma^s)d\mu - \int_MFH|A^o|^2\gamma^sd\mu
\notag\\*
  &\hskip-1.5cm + \int_M H\IP{\nabla F}{\nabla\gamma^s} - F\IP{\nabla H}{\nabla\gamma^s}d\mu,
\text{ and}
\notag
\end{align}
\vspace{-5mm}
\begin{align}
\rD{}{t}\int_M|A^o|^2\gamma^sd\mu + \int_M|F|^2\gamma^sd\mu
  &=
   \int_M |A^o|^2(\partial_t\gamma^s)d\mu + \int_MFH|A^o|^2\gamma^sd\mu
\notag\\*
  &\hskip-1.5cm + 2\int_M H\IP{A^o}{\nabla F\nabla\gamma^s} + F\IP{\nabla^* A^o}{\nabla\gamma^s}d\mu
\notag\\
  &\hskip-1.5cm - 2 \int_M F(A^o)^i_j(A^o)^j_k(A^o)^k_i \gamma^sd\mu.
\label{eqnS2lemgradest1eq2}
\end{align}
\end{lem}
\begin{proof}
Using Lemma \ref{lemS2evolutionequations},
\begin{align*}
\rD{}{t}\int_M\frac{1}{2}H^2\gamma^sd\mu
 &=  \int_M H(-\Delta F - |A|^2F)\gamma^sd\mu
   + \int_M \frac{1}{2}H^2(\partial_t\gamma^s)d\mu
   + \int_M \frac{1}{2}H^3F\gamma^sd\mu
\\
 &=  -\int_M H(\Delta F) \gamma^s d\mu - \int_MFH|A^o|^2\gamma^sd\mu 
  + \frac{1}{2}\int_MH^2(\partial_t\gamma^s)d\mu.
\end{align*}
Integrating by parts twice,
\begin{align*}
\rD{}{t}\int_M\frac{1}{2}H^2\gamma^sd\mu + \int_MF^2\gamma^sd\mu
 &=  \frac{1}{2}\int_MH^2(\partial_t\gamma^s)d\mu - \int_MFH|A^o|^2\gamma^sd\mu
\\*
 &\quad + \int_MH\IP{\nabla F}{\nabla\gamma^s} - F\IP{\nabla H}{\nabla\gamma^s}d\mu.
\end{align*}
This proves the first statement.  For the second, first compute
\begin{align*}
\rD{}{t}\int_M|A^o|^2\gamma^sd\mu
 &=  2\int_M \IP{A^o_{ij}}{-S^o_{ij}(\nabla_{(2)}F)
      + F\big[A_i^pA_{pj} + \sfrac{1}{n}g_{ij}|A|^2-\sfrac{2}{n}HA_{ij}\big]}\gamma^sd\mu
\\
 &\quad
   + \int_M |A^o|^2HF\gamma^sd\mu
   + \int_M |A^o|^2(\partial_t\gamma^s)d\mu
\\
 &= -2\int_M \IP{A^o}{\nabla_{(2)}F}\gamma^sd\mu
    -2\int_M F(A^o)^i_j(A^o)^j_k(A^o)^k_i \gamma^sd\mu
\\
 &\quad
   + \int_M |A^o|^2HF\gamma^sd\mu
   + \int_M |A^o|^2(\partial_t\gamma^s)d\mu,
\end{align*}
since $\IP{A^o_{ij}}{A_i^pA_{pj}  - \frac{2}{n}HA_{ij}} = (A^o)_{ij}(A^o)^i_k(A^o)^{kj}$.  Using now
$\nabla_{(2)}^*A^o = \frac{n-1}{n}\Delta H$
 and integration by parts twice
gives the second statement, and so we are finished.
\qed
\end{proof}

Considering for the moment the tracefree curvature estimate only, we tailor the right hand side of
the previous identity to be compatible with the multiplicative Sobolev inequalities of Lemma
\ref{lemS2msiboth}.

\begin{lem} 
Let $\gamma$ be as in \eqref{eqnS2eqngamma}, $s\ge4$, and $\delta_1,
\delta_2, \delta_3$ be fixed positive numbers.  The following inequality holds for a surface
diffusion flow $f:M^n\times[0,T)\rightarrow\R^{n+1}$:
\label{lemS2gradests2}
\begin{align*}
\rD{}{t}\int_M|A^o|^2\gamma^sd\mu + (1-\delta_1)\int_MF^2\gamma^sd\mu
  &\le 
   \delta_2\int_M|\nabla A^o|^2H^2\gamma^sd\mu\\
  &\hskip-4cm
         + \delta_3\int_M|A|^4|A^o|^2\gamma^sd\mu 
         + c\int_M(|A^o|^6 + |\nabla A^o|^2|A^o|^2)\gamma^sd\mu \\
  &\hskip-4cm
         + c\big[(c_{{\gamma}2})^2 + (c_{\gamma1})^4\big]
            \int_M|A^o|^2\gamma^{s-4}d\mu.
\end{align*}
where $c = c(s,n,\delta_i)$.
\end{lem}
\begin{proof}
We first compute, using integration by parts and the definition of $\gamma$,
\begin{align*}
2\int_M &(A^o)_{ij}(\nabla^i\gamma^s)(\nabla^jF)d\mu + 2\int_MF\IP{\nabla^*A^o}{\nabla\gamma^s}d\mu
  = -2\int_M F\IP{A^o}{\nabla_{(2)}\gamma^s}d\mu
\\
  &= -2s\int_MF\Big(\gamma\IP{A^o}{(D^2\tilde{\gamma}\circ f)g+(D\tilde{\gamma}\circ f)A}
                  + (s-1)\IP{A^o}{\nabla\gamma\nabla\gamma}\Big)\gamma^{s-2}d\mu.
\intertext{Note that $\IP{A^o}{(D^2\tilde{\gamma}\circ f)g} = \sum_{i,j=1}^n
A^o_{ij}(D_{ij}\tilde{\gamma})g_{ij}$ is not in general zero, as each term in
the sum is scaled by the second derivatives of $\tilde{\gamma}$.  Continuing,}
  &\hskip-6mm
2\int_M (A^o)_{ij}(\nabla^i\gamma^s)(\nabla^jF)d\mu + 2\int_MF\IP{\nabla^*A^o}{\nabla\gamma^s}d\mu
\\*
  &\le
   c\int_M|F|\big(
  |A^o|(c_{{\gamma}2})+|A^o|^2(c_{{\gamma}1})
                     \big)\gamma^{s-1}d\mu
 + c\int_M|F|\cdot|A^o|(c_{\gamma1})^2\gamma^{s-2}d\mu
\\
  &\le
   \frac{\delta_1}{2}\int_M|F|^2\gamma^sd\mu
 + \int_M|A^o|^6\gamma^{s}d\mu
 + c\delta_1^{-2}
              \big[(c_{{\gamma}2})^2 + (c_{\gamma1})^4 \big]
              \int_M|A^o|^2\gamma^{s-4}d\mu,
\end{align*}
We now estimate the terms $\int_M FH|A^o|^2\gamma^sd\mu$ and $\int_M
F(A^o)^i_j(A^o)^j_k(A^o)^k_i \gamma^sd\mu$ from Lemma \ref{lemS2gradests1}.  Using integration by
parts and 
$\nabla_{(2)}^*A^o = \frac{n-1}{n}\Delta H$,
\begin{align*}
\int_M FH|A^o|^2\gamma^sd\mu
  &=
-\int_M |\nabla H|^2|A^o|^2\gamma^sd\mu
-s\int_M H|A^o|^2\IP{\nabla H}{\nabla\gamma}\gamma^{s-1}d\mu
\\*
  &\quad
-2\int_M H|A^o|\IP{\nabla H}{\nabla|A^o|}\gamma^sd\mu
\\
  &\le c_{\delta_3}\int_M|\nabla A^o|^2|A^o|^2\gamma^sd\mu
    + \delta_3\int_M|\nabla A^o|^2H^2\gamma^sd\mu
\\
  &\quad + c_{\delta_2}(c_{\gamma1})^4\int_M|A^o|^2\gamma^{s-4}d\mu
    + \delta_2\int_M|A|^4|A^o|^2\gamma^sd\mu.
\end{align*}
This estimates the first integral.  The second is easily estimated by
\[
\int_MF(A^o*A^o*A^o)\gamma^sd\mu
 \le \frac{\delta_1}{4}\int_MF^2\gamma^sd\mu + c_{\delta_1}\int_M|A^o|^6\gamma^sd\mu.
\]
Finally we estimate the time derivative of $\gamma$ as follows.
\begin{align*}
\int_M|A^o|^2(\partial_t\gamma^s)d\mu
 &= 
s\int_M|A^o|^2(\gamma^{s-1}\IP{D\tilde{\gamma}}{\nu}F)d\mu
\\
 &\le \frac{\delta_1}{4}\int_M F^2\gamma^sd\mu +
c_{\delta_1}(c_{\tilde{\gamma}1})^4\int_M|A^o|^2\gamma^{s-4}d\mu + \int_M|A^o|^6d\mu.
\end{align*}
Combining these inequalities with Lemma \ref{lemS2gradests1} and absorbing finishes the proof.
\qed
\end{proof}

The following proposition demonstrates some local control of the $L^2$ norm of the tracefree
curvature.

\begin{prop}
Suppose $f:M^2\times[0,T^*]\rightarrow\R^{3}$ is a surface diffusion flow where for all $x\in\R^3$
there exists $\sigma < \infty$ such that $\vn{A}^2_{2,f^{-1}(B_\rho(x))}\le\sigma$.
Then there exist constants $\epsilon_1>0$ and $c_4 = c_4(\epsilon_0) > 0$ such that if
\begin{equation}
\sup_{x\in\R^3}\int_{f^{-1}(B_\rho(x))}|A^o|^2d\mu\bigg|_{t=0}
 < \epsilon_1
\label{eqnS2propapsmallness}
\end{equation}
then at any time $0 \le t < t_1 = \min\{c_4\sigma^{-1}\rho^4, T^*\}$ and for any $x\in\R^3$ we have
\begin{align}
\int_{f^{-1}(B_\rho(x))} |A^o|^2d\mu + \frac{1}{2}\int_0^t \int_{f^{-1}(B_\rho(x))}
&\big(|\nabla_{(2)}A|^2 + |\nabla A|^2|A|^2 + |A|^4|A^o|^2\big) d\mu d\tau
\notag\\*
&\le c[\epsilon_1+\sigma\rho^{-4}t].
\label{CieP8e3}
\end{align}
\label{propS2apest}
\end{prop}
\begin{proof}
\vspace{-0.7cm}
We first establish an estimate for the tracefree curvature.
Assuming that \eqref{eqnS2propapsmallness} is satisfied on $[0,t]$, $t<t_1$, we combine Proposition
\ref{propS2GLprop5} with the multiplicative Sobolev inequality \eqref{eqnS2msi1Azero} and Lemma
\ref{lemS2gradests2} to obtain
\begin{align*}
\rD{}{t}\int_M|A^o|^2\gamma^sd\mu
 &+ c(1-\text{$\textstyle\sum\delta_i$}- c_2\epsilon_0)
 \int_M(|\nabla_{(2)}A|^2 + |\nabla A|^2|A|^2 + |A|^4|A^o|^2)\gamma^sd\mu
\\
  &\le 
   \frac{c}{\rho^4}\vn{A^o}^4_{2,[\gamma>0]} + \frac{c}{\rho^4} \vn{A}^2_{2,[\gamma>0]}
  \le 
   \frac{c}{\rho^4} \vn{A}^2_{2,[\gamma>0]}.
\end{align*}
where $\epsilon_0 > 0$ is as in Proposition \ref{propS2GLprop5}.
Choosing $\delta_i$ small enough and requiring $\epsilon_0$ to be such that
$c(1-\sum_i\delta_i-c_2\epsilon_0) = \frac{1}{2}$
we have shown
\begin{equation}
\rD{}{t}\int_M|A^o|^2\gamma^sd\mu
  + \frac{1}{2}\int_M(|\nabla_{(2)}A|^2 + |\nabla A|^2|A|^2 + |A|^4|A^o|^2)\gamma^sd\mu
  \le \frac{c}{\rho^4} \vn{A}^2_{2,[\gamma>0]}.
\label{eqnS2apesteq1}
\end{equation}
Now integrate \eqref{eqnS2apesteq1} to obtain
\begin{align}
\int_{f^{-1}(B_{\rho/2}(x))}\hskip-5mm|A^o|^2d\mu
&+ \frac{1}{2}\int_0^t\int_{f^{-1}(B_{\rho/2}(x))}\hskip-5mm 
|\nabla_{(2)}A|^2 + |\nabla A|^2|A|^2 + |A|^4|A^o|^2 d\mu d\tau
\notag\\*
&\le \epsilon_1 + c\sigma\rho^{-4}t.
\label{CieP8e2}
\end{align}
Motivated by the fact that $2^{4}$ balls $B_{\rho/2}$ can be used to cover a ball $B_{\rho}$, we
require $\epsilon_1 \le \frac{\epsilon_0}{4\cdot2^{4}}$.
Since
$0 < t \le \frac{\epsilon_0}{c4\cdot2^4}\rho^4\sigma^{-1} = c_4\sigma^{-1}\rho^4$,
we use \eqref{CieP8e2} and a covering argument to derive
\begin{align*}
\int_{f^{-1}(B_{\rho}(x))}|A^o|^2d\mu
&+ \frac{1}{2}\int_0^t\int_{f^{-1}(B_{\rho}(x))}
|\nabla_{(2)}A|^2 + |\nabla A|^2|A|^2 + |A|^4|A^o|^2 d\mu d\tau
\\
&\le 2^4\Big(\epsilon_0\frac{1}{4\cdot2^4} + \epsilon_0\frac{1}{4\cdot2^4}\Big)
 \le \frac{\epsilon_0}{2}.
\end{align*}
Therefore the smallness assumption of Proposition \ref{propS2GLprop5} holds up to time $t=t_1$ and
\eqref{CieP8e3} follows.
\qed
\end{proof}

For the interested reader we also state an easy corollary.

\begin{cor}
Under the assumptions of Proposition \ref{propS2apest} above,
\begin{align}
\int_0^t\vn{A^o}^4_{\infty,f^{-1}(B_\rho(x_1))} d\tau 
  &\le c[\epsilon_1 + \sigma\rho^{-4}t],\text{ and}
\notag
\\
\int_{f^{-1}(B_{\rho'/2}(x_1))} |A|^2 d\mu \bigg|_{t=\tau}
&\le 
\int_{f^{-1}(B_{\rho'}(x_1))} |A|^2 d\mu \bigg|_{t=0} + c(\sigma(\rho')^{-4})\tau,
\notag
\end{align}
where $0<\rho'\le\rho$ and $\tau\le\min\{c_4\sigma^{-1}(\rho')^4,T^*\}$.
\end{cor}
\begin{proof}
To obtain the first estimate, combine \eqref{CieP8e3} with Lemma \ref{lemS2int} and use a covering
argument.  For the second, first note that an argument completely analagous to that used to obtain
\eqref{eqnS2apesteq1} also gives
\[
\rD{}{t}\int_MH^2\gamma^sd\mu
  + \frac{1}{2}\int_M(|\nabla_{(2)}A|^2 + |\nabla A|^2|A|^2 + |A|^4|A^o|^2)\gamma^sd\mu
  \le \frac{c}{\rho^4} \vn{A}^2_{2,[\gamma>0]}.
\]
Now the second statement follows from integrating this and combining with the estimate of
Proposition \ref{propS2apest}.
\qed
\end{proof}

We now turn to controlling the higher derivatives of curvature.  Using Lemma
\ref{lemS2evolutionequations}, one may derive the following local integral estimate.  The proof is
long but somewhat standard; see \cite{kuwert2002gfw} for the case of Willmore flow.  One computes
the time derivative of $\vn{(\nabla_{(k)}A)\gamma^s}^2_2$, integrates by parts, estimates the
result, interpolates, and absorbs higher derivative terms on the left.

\begin{prop}
\label{propS2evolutionest1}
Let $f:M^2\times[0,T)\rightarrow\R^{3}$ be a surface diffusion flow and $\gamma$ 
as in \eqref{eqnS2eqngamma}.  Then for a fixed $\theta > 0$ and $s\ge2k+4$,
\begin{align}
&\rD{}{t}\int_M |\nabla_{(k)}A|^2\gamma^sd\mu
 + (2-\theta)\int_M |\nabla_{(k+2)}A|^2\gamma^sd\mu\notag\\*
&\qquad \le c\int_M |A|^2\gamma^{s-4-2k}d\mu
            + c\int_M \left([P_3^{k+2}(A)+P_5^k(A)]*\nabla_{(k)}A\right)\gamma^{s} d\mu,
\notag
\end{align}
where $c=c(c_{\gamma1},c_{\gamma2},s,k,\theta)$.
\end{prop}

Note that in the above local smallness of curvature is not assumed.  Using this we obtain the
following.

\begin{prop} 
\label{propS2evolutionest2}
Suppose $f:M^2\times[0,T^*]\rightarrow\R^{3}$ is a surface diffusion flow and $\gamma$
is as in \eqref{eqnS2eqngamma}.  Then there is an $\epsilon_0>0$ such that if
\begin{equation*}
\epsilon = \sup_{[0,T^*]}\int_{[\gamma>0]}|A|^2d\mu\le\epsilon_0
\end{equation*}
then for any $t\in[0,T^*]$ we have
\begin{equation}
\label{eq10}
  \begin{split}
&\int_{[\gamma=1]} |A|^2 d\mu + \int_0^t\int_{[\gamma=1]} (|\nabla_{(2)}A|^2 +
|A|^2|\nabla A|^2 + |A|^6) d\mu d\tau \\
&\qquad\qquad\qquad \le \int_{[\gamma > 0]} |A|^2 d\mu\Big|_{t=0}
 + c\epsilon t,
  \end{split}
\end{equation}
where $c=c(c_{\gamma1},c_{\gamma2})$.
\end{prop}

The idea of the proof is to integrate Proposition \ref{propS2evolutionest1}, and then use the
multiplicative Sobolev inequality \eqref{eqnS2msi2A}.  This will introduce a multiplicative factor
of $\vn{A}^2_{2,[\gamma>0]}$ in front of several integrals, which we can then absorb on the left.
Since this is similar to \cite{kuwert2002gfw}, we omit the proof.

We will need the below proposition to obtain interior estimates.  This is proved using
Proposition \ref{propS2evolutionest1}, interpolating the $P$-terms using Corollary 5.5 from
\cite{kuwert2002gfw} and then absorbing.

\begin{prop}
\label{p:prop45}
Suppose $f:M^2\times[0,T]\rightarrow\R^{3}$ is a surface diffusion flow and
$\gamma$ a localisation function as in \eqref{eqnS2eqngamma}.  Then, for $s \ge 2k+4$
the following estimate holds:
\begin{equation}
\label{e:prop45}
  \begin{split}
&\rD{}{t}\int_{M} |\nabla_{(k)}A|^2\gamma^s d\mu
 + \int_{M} |\nabla_{(k+2)}A|^2 \gamma^s d\mu \\
&\quad
  \le c\vn{A}^4_{\infty,[\gamma>0]}\int_M\vn{\nabla_{(k)}A}^2\gamma^sd\mu
       + c\vn{A}^2_{2,[\gamma>0]}(1+\vn{A}^4_{\infty,[\gamma>0]}).
  \end{split}
\end{equation}
\end{prop}
We next state local interior estimates for \eqref{SD}, which hold where and when curvature is
locally small.  Once the estimates \eqref{eq10} and \eqref{e:prop45} are established, the proof
follows similarly to Theorem 3.5 in \cite{kuwert2001wfs}.  Briefly, one utilises cutoff functions in
time as well as the familiar spatial localisation functions $\gamma$ to obtain integral estimates of
Gronwall type, which in turn gives the estimates in $L^2$.  Then, using interpolation (such as in
Lemma \ref{lemS2int}), one obtains the $L^\infty$ estimates.

\begin{thm}[Interior estimates]
Suppose $f:M^2\times(0,T^*]\rightarrow\R^{3}$ is a surface diffusion flow satisfying
\begin{equation*}
\sup_{0<t\le T^*} \int_{f^{-1}(B_\rho(x))} |A|^2d\mu \le \epsilon_0,
\text{ for } T^* \le c\rho^4.
\end{equation*}
Then for any $k\in \N_0$ we have at time $t\in(0,T^*]$ the estimates
\begin{align*}
\vn{\nabla_{(k)}A}_{2,f^{-1}(B_{\rho/2}(x))} &\le c_k\sqrt{\epsilon_0}t^{-\frac{k}{4}}
\\
\vn{\nabla_{(k)}A}_{\infty,f^{-1}(B_{\rho/2}(x))} &\le c_k\sqrt{\epsilon_0}t^{-\frac{k+1}{4}},
\end{align*}
where $c_k = c_k\big(k, \rho, T^*,
\vn{\nabla_{(k)}A}_{2,f^{-1}(B_\rho(x_0))}|_{t=0}\big)$.
\label{thmS2ie}
\end{thm}

Armed with these estimates, we may conclude the following concentration-compactness alternative.
The proof is similar to that of Theorem 1.2 in \cite{kuwert2002gfw}: one defines a time parameter
$t_0$ and corresponding interval $[0,t_0]$, which is the time during which the curvature is locally
small in $L^2$.  Using a continuity argument, one shows that either $t_0 = c_0\rho^4$, in which case
one can conclude the theorem, or that $t_0 = T$.  For the latter case, we then use the interior
estimates and a tiling argument such as can be found in \cite{RH} or \cite{huisken1984fmc} to show
that $T=\infty$, and we again conclude the theorem.

\begin{thm}[Lifespan Theorem]
\label{thmS2lifespan}
Suppose $f:M^2\times[0,T)\rightarrow\R^{3}$ is a surface diffusion flow with $f_0$ smooth.
Then there are constants $\rho, \epsilon_0, c_0>0$
such that if $\rho$ is chosen with
\begin{equation*}
\int_{f^{-1}(B_\rho(x))} |A|^2 d\mu\Big|_{t=0} \le \epsilon_0
\qquad
\text{ for any $x\in\R^{3}$},
\end{equation*}
then the maximal time $T$ satisfies
\begin{equation*}
T \ge c_0\rho^4,
\end{equation*}
and we have the estimate
\begin{equation*}
\int_{f^{-1}(B_\rho(x))} |A|^2 d\mu \le c\epsilon_0
\qquad\text{ for }\qquad
0\le t \le c_0\rho^4.
\end{equation*}
\end{thm}

\section{Blowup analysis}

This section is split into several parts.  First we state the required compactness theorem, and
detail the construction of the blowup.  Our primary concern is to demonstrate that under suitable
conditions the blowup is a stationary solution of surface diffusion flow.
Although we do not enjoy an explicit gradient flow of curvature, Proposition
\ref{propMonotoneCurvature} tells us that surface diffusion flow is `almost' a gradient flow of
curvature, in the sense that if the tracefree curvature is initially globally small enough, this
condition is preserved and in fact the $L^2$ norm of the curvature is nonincreasing.
Combining this with our estmates from Section 2, we are able to essentially follow
\cite{kuwert2001wfs}.  For the convenience of the reader, we detail the entire argument.

Observe that by Theorem \ref{thmS2lifespan}, $T<\infty$ implies that the curvature has
concentrated in $L^2$ (see \eqref{eqnS3curvconc}).  Assuming
\eqref{eqS1maintheoremsmallnessrequirement}, we prove that the blowup at a time where the
curvature has concentrated in $L^2$ is a stationary, noncompact, nonumbilic surface satisfying the
assumptions of Theorem \ref{thmS2gaplemma}.
We apply this in the next section, where we contradict finite maximal existence time with Theorem
\ref{thmS2gaplemma} to conclude long time existence.  A further argument strengthens this result to
exponential convergence to spheres, and this finishes the proof of Theorem \ref{thmS1maintheorem}.

\begin{thm}[Theorem 4.2, \cite{kuwert2001wfs}]
\label{C7Tblowup}
Let $f_j:M_j\rightarrow\R^3$ be a sequence of proper immersions, where $M_j$ is a surface without
boundary.  Let
\[
M_j(R) = \{p\in M_j: |f_j(p)| < R\}
\]
and assume the bounds
\begin{align*}
\mu_j(M_j(R)) &\le c(R)\ \ \text{for any $R>0$,}
\\
\vn{\nabla_{(k)}A}_{\infty,M_j} &\le c(k)\ \ \text{for any $k\in\N_0$.}
\end{align*}
Then there exists a proper immersion $\tilde{f} : \tilde{M}\rightarrow\R^3$, where $\tilde{M}$ is
again a surface without boundary, such that after passing to a subsequence we have a representation
\[
  f_j\circ \phi_j = \tilde{f}+u_j\ \ \text{on}\ \ \tilde{M}(j) = \{p\in \tilde{M}: |\tilde{f}(p)|<j\}
\]
with the following properties:
\begin{align}
&\phi_j:\tilde{M}(j) \rightarrow U_j\subset M_j\text{ is diffeomorphic,}
\notag\\
&M_j(R) \subset U_j\text{ if }j\ge j(R),
\notag\\
&u_j\in C^\infty(\tilde{M}(j),\R^3)\text{ is normal along $\tilde{f}$,}
\notag\\
&\vn{\tilde{\nabla}_{(k)}u_j}_{\infty,\tilde{M}(j)}\rightarrow 0\text{  as  }j\rightarrow\infty
\text{ for any $k\in\N_0$}.
\label{C7TblowupE1}
\end{align}
\end{thm}
The theorem says that on any ball $B_R(0)$ the immersion $f_j$ can be written as a normal graph with
small norm for $j$ large over a limit immersion $\tilde{f}$, after suitably reparametrising with
$\phi_j$.

Let $f:M^2\times[0,T)\rightarrow\R^3$ be a surface diffusion flow defined on a closed surface $M^2$,
where $0 < T \le \infty$.  Define
\[
\eta(r,t) = \sup_{x\in\R^3} \int_{f^{-1}(B_r(x))}|A|^2d\mu.
\] 
Let $r_j$ be an arbitrary decreasing sequence with $r_j \searrow 0$ and assume curvature
concentrates in the sense that for each $j$,
\begin{equation}
t_j = \inf\{t\ge0:\eta(r_j,t) > \epsilon_1\} < T,
\label{eqnS3curvconc}
\end{equation}
where $\epsilon_1 = \epsilon_0c_0$ and $\epsilon_0, c_0$ are as in the Lifespan Theorem.

\begin{lem}With the definitions above, we have
\label{C7L1}
\[
\int_{f^{-1}(B_{r_j}(x))}|A|^2d\mu\bigg|_{t=t_j} \le \epsilon_1
\text{  for any }x\in\R^3,
\]
and
\[
\int_{f^{-1}\overline{(B_{r_j}(x_j))}}|A|^2d\mu\bigg|_{t=t_j} \ge \epsilon_1
\text{  for some }x_j\in\R^3.
\]
\end{lem}
\begin{proof}
The first statement is a direct consequence of the definition of $t_j$.  For the second, fix $j$ and
consider a sequence $\nu \rightarrow \infty$, $\nu>1$.  Consider times $t_{j+\nu^{-1}} \searrow t_j$
and radii $r_{j+\nu^{-2}}\nearrow r_j$.  Then for each $\nu$ there exists an $x_{j+\nu^{-1}}$ such
that
\[
\int_{f^{-1}(B_{r_{j+\nu^{-2}} }(x_{j+\nu^{-1}}))}|A|^2d\mu\bigg|_{t=t_{j+\nu^{-1}} } \ge \epsilon_1.
\]
Taking $\nu\rightarrow\infty$ in the above equation gives the second statement.
\qed
\end{proof}
We now rescale $f$.  Define immersions
\[
f_j:M^2\times\big[-r_j^{-4}t_j, r_j^{-4}(T-t_j)\big)\rightarrow\R^3,
\qquad
f_j(p,t) = \frac{1}{r_j}\big(f(p,t_j+r_j^4t)-x_j\big).
\]
The sequence of immersions $f_j$ can be thought of as `zooming in' on the assumed curvature
singularity at time $T$.  Let $\eta_j(r,t)$ be $\eta$ with respect to the immersion $f_j$.  Then
from Lemma \ref{C7L1} we have $\eta_j(1,t) \le \epsilon_1$ for $t\le0$ and
\[
\int_{f_j^{-1}\overline{(B_1(0))}} |A|^2d\mu\bigg|_{t=0} \ge \epsilon_1.
\]
The Lifespan Theorem implies $r_j^{-4}(T-t_j) \ge c_0$ for any $j$ and also that
\begin{equation}
\eta_j(1,t) \le \epsilon_0\quad\text{for}\quad0<t\le c_0.
\label{C7E6}
\end{equation}
We apply the interior estimates on parabolic cylinders $B_1(x)\times(t-1,t]$ to obtain
\begin{equation}
\vn{\nabla_{(k)}A}_{\infty,f_j} \le c(k)\quad\text{for}\quad-r_j^{-4}t_j+1\le t\le c_0.
\label{C7E2}
\end{equation}
Surface area is uniformly bounded, from above and below.  Note that in light of Proposition
\ref{propMonotoneCurvature} we could also follow \cite{kuwert2001wfs} and use the local area
estimate of Simon \cite{simon1993esm}.
Using Theorem \ref{C7Tblowup} with the sequence $f_j = f_j(\cdot,0):M^2\rightarrow\R^3$ we obtain a
limit immersion
$\tilde{f}_0:\tilde{M}\rightarrow\R^3$.  Let $\phi_j:\tilde{M}(j)\rightarrow
U_j\subset M^2$ be as in \eqref{C7TblowupE1}.  Then the reparametrisation
\begin{equation}
  f_j(\phi_j,\cdot):\tilde{M}(j)\times[0,c_0]\rightarrow\R^3
\label{C7E3}
\end{equation}
is a surface diffusion flow with initial data
\begin{equation*}
f_j(\phi_j,0) = \tilde{f}_0+u_j:\tilde{M}(j)\rightarrow\R^3.
\end{equation*}
The flows \eqref{C7E3} satisfy the curvature bounds \eqref{C7E2} and have initial data converging
locally in $C^k$ to the immersion $\tilde{f}_0:M\rightarrow\R^3$.  By converting the curvature
bounds to partial derivative bounds in parabolic cylinders we obtain the locally smooth convergence
\begin{equation}
f_j(\phi_j,\cdot) \rightarrow \tilde{f},
\label{C7E5}
\end{equation}
where $\tilde{f}:\tilde{M}\times[0,c_0]\rightarrow\R^3$ is a surface diffusion flow with initial
data $f_0$.

\begin{thm}
Let $f:M^2\times[0,T)\rightarrow\R^3$ be a surface diffusion flow satisfying
\eqref{eqS1maintheoremsmallnessrequirement}.
Then the blowup $\tilde{f}$ as constructed above is stationary.
\label{thmS3bu}
\end{thm}
\begin{proof}
Noting the scale invariance of $\vn{A}^2_2$, we use Proposition \ref{propMonotoneCurvature} to
compute
\begin{align*}
\int_0^{c_0}\int_{\tilde{M}(j)}&|\Delta H(f_j(\phi_j,t))|^2d\mu_{f_j(\phi_j,\cdot)}dt
= \int_0^{c_0}\int_{U_j}|\Delta H_j|^2d\mu_jdt
\\
&\le
   \int_M |A_j(0)|^2d\mu_j
 - \int_M |A_j(c_0)|^2d\mu_j
\\
&= 
   \int_M |A(t_j)|^2d\mu
 - \int_M |A(t_j+r_j^4c_0)|^2d\mu,
\end{align*}
and this converges to zero as $j\rightarrow\infty$.  Therefore $\Delta H(\tilde{f}) \equiv 0$ and
the blowup $\tilde{f}$ is stationary under surface diffusion flow.
\qed
\end{proof}

\begin{rmk}  The above is strictly weaker than the corresponding result for Willmore flow from
\cite{kuwert2001wfs}.  In particular, we need the tracefree curvature to be small at initial time to
obtain a stationary blowup at a finite time curvature singularity, whereas for Willmore flow the
corresponding blowup is always a stationary Willmore surface.  Note that this means one may still
obtain self-similar solutions from a blowup of surface diffusion flow, which one might even expect
given initial data with zero enclosed volume, such as a symmetric figure eight.
\end{rmk}

We now need to show that $\tilde{f}$ is nontrivial.  The arguments below follow
\cite{kuwert2001wfs}, although here the area bound does not present any difficulty.

\begin{lem}
The blowup $\tilde{f}$ constructed above is not a union of planes.
\label{corblowup}
\end{lem}
\begin{proof}
Due to the smooth convergence in \eqref{C7E5} and the second conclusion in Lemma \ref{C7L1} we have
\[
\int_{\tilde{f}^{-1}\overline{(B_1(0))}} |A|^2d\mu \ge \epsilon_1 > 0.
\]
\qed
\end{proof}

\begin{lem}
If the blowup $\tilde{f}$ constructed above contains a compact connected component $C$, then
$\tilde{M} = C$ and $M$ is diffeomorphic to $C$.
\label{C7L2}
\end{lem}
\begin{proof}
For $j$ sufficiently large, $\phi_j(C)$ is open and closed in $M$.  By the connectedness of $M$ we
have $M=\phi_j(C)$ and thus $\tilde{M} = C$.
\qed
\end{proof}

\begin{thm}[Nontriviality of the blowup]
Suppose $f:M^2\times[0,T)\rightarrow\R^3$ is a surface diffusion flow satisfying
\eqref{eqS1maintheoremsmallnessrequirement}
and let $\tilde{f}$ be the blowup constructed above.  Then none of the components of $\tilde{f}$
are compact.  In particular, the blowup has a component which is a noncompact nonumbilic surface
with $\Delta H \equiv 0$.
\label{BUnontrivial}
\end{thm}
\begin{proof}
Observe that the surface area $\mu(M)$ is uniformly bounded away from zero, first by noting that each
$f_t$ is an embedding (by Proposition \ref{propMonotoneCurvature}) and then by combining the
isoperimetric inequality with the computation \eqref{eqvolcomp}.

Assume that there is a compact component of $\tilde{f}$.  Then Lemma \ref{C7L2} implies that
$\tilde{f}:\tilde{M}\rightarrow\R^3$ has no further components. 
Therefore the surface area of the blowup is bounded.  The measure behaves under scaling by
\[
\mu(M)\Big|_{t=t_j} = r_j^2\mu_j(M)\Big|_{t=0}
\]
and so we have
\[
\mu(M)\Big|_{t=T} = \lim_{j\rightarrow\infty}\mu(M)\Big|_{t=t_j}
                  = \lim_{j\rightarrow\infty}r_j^2\mu_j(M)\Big|_{t=0}
                  = 0.
\]
This is in direct contradiction with the fact that area is uniformly bounded away from zero.  Thus
there are no compact components of $\tilde{f}$, and Lemma \ref{corblowup} gives that there must
exist at least one nonumbilic noncompact component of $\tilde{f}$.
\qed
\end{proof}

\section{Asymptotic behaviour}

Combining the analysis of the previous sections we can finally rule out concentration of curvature
in finite time.

\begin{prop}
Suppose $f:M^2\times[0,T)\rightarrow\R^3$ is a \eqref{SD} flow.  Then there exists a
constant $\epsilon_0>0$ such that if
\[
\int_M|A^o|^2d\mu\bigg|_{t=0} < \epsilon_0 
\]
then $T=\infty$.
\label{propS4LTE}
\end{prop}
\begin{proof}
Assume otherwise, and then by the Lifespan Theorem there exists a $T < \infty$ such that curvature
concentrates at time $T$. Performing a blowup as in Section 3 at $T$ we recover a stationary surface
$\tilde{f}$ with small tracefree curvature due to the scale invariance of $\vn{A^o}_2^2$ and
Proposition \ref{propMonotoneCurvature}.  Now Theorem \ref{thmS2gaplemma} implies $\tilde{f}$ must
be a plane or sphere, which contradicts the nontriviality of the blow up, Theorem
\ref{BUnontrivial}.  Thus there does not exist a finite time when curvature concentrates, and so
$T=\infty$.
\qed
\end{proof}

Smooth (sub)convergence to round spheres now follows by using a similar argument, constructively
instead of for the purposes of contradiction.  This is similar to Lemma 5.4 in \cite{kuwert2001wfs},
and so we omit the proof.

\begin{lem}
Suppose $f:M^2\times[0,T)\rightarrow\R^3$ is a surface diffusion flow satisfying
\eqref{eqS1maintheoremsmallnessrequirement}.  Then for any sequence $t_j\nearrow\infty$ there exist
$x_j\in\R^3$ and $\phi_j\in\text{Diff}(M)$ such that, after passing to a subsequence, the immersions
$f(\phi_j,t)-x_j$ converge smoothly to an embedded round sphere.
\label{lemS4smoothconv}
\end{lem}

The above implies that Proposition \ref{propMonotoneCurvature} may in fact be strengthened to
\[
\int_M |A^o|^2 d\mu \searrow 0\text{ as }t\nearrow\infty.
\]
We finish by proving exponential decay of curvature.

\begin{prop}
Suppose $f:M^2\times[0,T)\rightarrow\R^3$ is a surface diffusion flow satisfying
\eqref{eqS1maintheoremsmallnessrequirement}.  Then there exists a $\lambda > 0$ such that as
$t\nearrow\infty$ the following asymptotic statements hold:
\begin{equation*}
\vn{\nabla_{(k)}A}_\infty \le c^ke^{-\lambda t},
\quad\text{and}\quad
\vn{A^o}_\infty \le c^0e^{-\lambda t},
\end{equation*}
for $k\ge 1$.
\end{prop}
\begin{proof}
First note that using an argument similar to that used to prove \eqref{eqnS2apesteq1} one may
establish the estimate
\begin{equation}
\rD{}{t}\int_M |A^o|^2d\mu
 + \frac{1}{100}\int_M \big(|\nabla_{(2)}A|^2 + |\nabla A|^2H^2 + |A^o|^2H^4\big)d\mu \le 0.
\label{eqnS4expconvest}
\end{equation}
Let $\SA = \mu(M_t)\big|_{t=T}$.  Then Lemma \ref{lemS4smoothconv} above implies that the sectional
curvature and mean curvature satisfy
\[
\vn{K}_\infty \longrightarrow \frac{4\pi}{\SA},\quad\text{and}\quad
\vn{H^2}_\infty \longrightarrow  \frac{16\pi}{\SA}
\]
as $t\nearrow\infty$ respectively.  Therefore there exists a $t_H < \infty$ such that
\[
H^2 \ge c_H > 0\text{,  for all }t\ge t_H.
\]
From now on we assume $t\ge t_H$.  Note that we may assume $c_H < 1$.  Using
\eqref{eqnS4expconvest},
\[
\rD{}{t}\int_M|A^o|^2d\mu + \frac{c_H^2}{100}\int_M|\nabla_{(2)}A|^2
 + |\nabla A|^2+|A^o|^2d\mu \le 0.
\]
Integrating gives
\begin{equation}
\int_M |A^o|^2d\mu
 + \frac{c_H^2}{100}\int_t^\infty\int_M|\nabla_{(2)}A|^2+|\nabla A|^2d\mu d\tau
\le e^{-2\lambda t},
\label{eqnEXPconvE1}
\end{equation}
where
$
\lambda = \frac{c_H^2}{200}.
$
Note that we needed Lemma \ref{lemS4smoothconv} for $\int_M|A^o|^2d\mu\big|_{t=T} = 0$ in the
above.  From this estimate and again Lemma \ref{lemS4smoothconv} we can obtain exponential decay.
From Proposition \ref{propS2evolutionest1}, taking $\rho\nearrow\infty$ gives
\[
\rD{}{t}\int_M|\nabla_{(k)}A|^2d\mu + \int_M|\nabla_{(k+2)}A|^2d\mu
\le \int_M \big(P_3^{k+2}(A)+P_5^k(A)\big)*\nabla_{(k)}Ad\mu.
\]
Note that the term $c\vn{A}^2_{2,[\gamma>0]}$ disappeared due to the dependence of the constant on
$\rho$.  From Lemma \ref{lemS4smoothconv} we know that $A$ and all its derivatives remain bounded
as $t\nearrow\infty$, so we estimate
\begin{align*}
\int_M P_2^0(A)*\nabla_{(k+2)}A*\nabla_{(k)}Ad\mu
 &\le \epsilon\int_M|\nabla_{(k+2)}A|^2d\mu + c_\epsilon\int_M|\nabla_{(k)}A|^2d\mu,
\\
\int_M \big(\tilde{P}_3^{k+2}(A)+P_5^m(A)\big)*\nabla_{(k)}Ad\mu
 &\le c\sum_{j=1}^{k+1}\int_M|\nabla_{(j)}A|^2d\mu,
\end{align*}
where the constant is not universal (i.e. depends on derivatives of curvature).
$\tilde{P}_3^{k+2}(A)$ denotes all terms of type $P_3^{k+2}(A)$ that do not contain the $(k+2)$-th
derivative.
We thus obtain
\[
\rD{}{t}\int_M|\nabla_{(k)}A|^2d\mu + \frac{1}{2}\int_M|\nabla_{(k+2)}A|^2d\mu
 \le c\sum_{j=1}^{k+1}\int_M|\nabla_{(j)}A|^2d\mu.
\]
A proof by induction using \eqref{eqnEXPconvE1} then gives
\[
\vn{\nabla_{(k)}A}^2_2+\frac{c_H^2}{100}\int_t^\infty\vn{\nabla_{(k+2)}A}^2_2d\tau
\le
e^{-2\lambda t}.
\]
This gives us the estimates
\[
\vn{A^o}_2 \le ce^{-\lambda t}\text{, and }\vn{\nabla_{(k)}A}_2 \le ce^{-\lambda t}.
\]
Using Lemma \ref{lemS2int} finishes the proof.
\qed
\end{proof}

\bibliographystyle{plain}
\bibliography{mbib}

\end{document}